\documentclass[a4paper,12pt,reqno]{article}
\usepackage{amsmath,amssymb,amsthm}
\usepackage{enumerate}
\usepackage{pxfonts}
\usepackage{color,graphicx}
\usepackage{authblk}
\usepackage{mathrsfs}
\usepackage{latexsym}

\usepackage[top=1truein,bottom=1truein,left=1truein,right=1truein]{geometry}
\usepackage{setspace}
\usepackage{geometry}

\usepackage[numbers,sort&compress]{natbib}
\usepackage{upgreek}
\usepackage{hyperref}

\newtheorem{theorem}{Theorem}[section]
\newtheorem{proposition}[theorem]{Proposition}

\newtheorem{lemma}[theorem]{Lemma}

\newtheorem{remark}[theorem]{Remark}

\makeatletter
\@addtoreset{equation}{section}

\makeatother
\title{Note on Error Bound {for} Trace Approximation of Products of Toeplitz Matrices}
\author
{Tetsuya Takabatake\footnote{tkbtk@hiroshima-u.ac.jp}}
\affil
{School of Economics, Hiroshima University\footnote{2-1 Kagamiyama 1-Chome, Higashi-Hiroshima, Hiroshima, Japan}}
\date{\today}

\begin{document}
\maketitle

\begin{abstract}
	We investigate error orders for integral limit approximations to traces of products of Toeplitz matrices generated by integrable functions on $[-\pi,\pi]$ having some singularities at the origin. 
	Even though a sharp error order of the above approximation is derived in Theorem~2 of \cite{Lieberman-Phillips-2004}, its proof contains an inaccuracy as pointed out by \cite{Ginovyan-Sahakyan-2013}. 
	In the present paper, we reinvestigate the claim given in Theorem~2 of \cite{Lieberman-Phillips-2004} and give an alternative proof of their claim. 	
\end{abstract}
\vspace{0.2cm}
\section{Introduction}
Let $p\in\mathbb{N}$, $\Pi:=[-\pi,\pi]$ and $\Pi_{0}:=\Pi\setminus\{0\}$. Denote by $L^{1}(\Pi)$ the set of integrable functions on $\Pi$ to $[-\infty,\infty]$. 
	For $f\in L^{1}(\Pi)$, we denote by $\widehat{f}(\tau)$ the $\tau$-th Fourier coefficient of $f$ defined by
	\begin{equation*}
		\widehat{f}(\tau):=\int_{\Pi}e^{\sqrt{-1}\tau x}f(x)\,\mathrm{d}x,\ \ \tau\in\mathbb{Z},
	\end{equation*}	
	and by $T_n(f)$ the $n\times n$-Toeplitz matrix whose $(i,j)$-element is given by $\widehat{f}(i-j)$ for each $i,j=1,\cdots,n$. 
	For $\alpha\in(-\infty,1)$, we introduce the function spaces $F_{\alpha}$ and $F_{\alpha}^{(1)}$ defined by
	\begin{align}
		&F_{\alpha}:=\left\{f\in L^{1}(\Pi):\sup_{x\in\Pi_{0}}|x|^{\alpha}|f(x)|<\infty\right\},\label{F-alpha}\\ 
		&F_{\alpha}^{(1)}:=\left\{f\in F_{\alpha}\cap C^{1}_{0}(\Pi):\sup_{x\in\Pi_{0}}|x|^{\alpha+1}\left|\frac{\mathrm{d}f}{\mathrm{d}x}(x)\right|<\infty\right\},\nonumber
	\end{align}
	where $C^{1}_{0}(\Pi)$ denotes the set of functions on $\Pi$ which are {continuously} differentiable on $\Pi_{0}$. 
Moreover, we write $x_{+}:=\max\{x,0\}$ for $x\in\mathbb{R}$. 
	
In the present paper, we will investigate on {the} convergence rate of the approximation error $\bar{E}_{n}^{p}$ defined by 
\begin{equation*}
	\bar{E}_{n}^{p}:=\left|\frac{1}{n}\mathrm{Tr}\left[\left(T_{n}(g)T_{n}(h)\right)^{p}\right]
	   		-(2\pi)^{2p-1}\int_{\Pi}\left[g(x)h(x)\right]^{p}\,\mathrm{d}x\right|
\end{equation*}
for functions $g\in{F}_{\alpha}$ and $h\in{F}_{\beta}$ with $\alpha,\beta\in(-\infty,1)$ under the following restriction {on} $\alpha$ and $\beta$:
\begin{equation}\label{para_rest1}
	p(\alpha+\beta)<1.
\end{equation}
Note that (\ref{para_rest1}) is necessary {for the integral in the definition of $\bar{E}_{n}^{p}$ to be finite.}

{The seminal work~\cite{Fox-Taqqu-1987} proved the convergence}
\begin{equation}\label{Error_Fox-Taqqu}
	\bar{E}_{n}^{p}=o(1)\ \ \mbox{as $n\to\infty$}
\end{equation}
for almost everywhere continuous even functions $g\in{F}_{\alpha}$ and $h\in{F}_{\beta}$, see Theorem~1 of \cite{Fox-Taqqu-1987}.  
After the work of \cite{Fox-Taqqu-1987}, \cite{Lieberman-Phillips-2004} investigated the convergence rate of $\bar{E}_{n}^{p}$ for even functions $g\in{F}_{\alpha}^{(1)}$ and $h\in{F}_{\beta}^{(1)}$ and claimed the following stronger version of (\ref{Error_Fox-Taqqu}) holds:
\begin{equation}\label{Error_Lieberman-Phillips}
	\bar{E}_{n}^{p}=o(n^{-1+p(\alpha+\beta)_{+}+\epsilon})\ \ \mbox{as $n\to\infty$}
\end{equation}
for any $\epsilon>0$, see Theorem~2 of \cite{Lieberman-Phillips-2004}. 
Unfortunately, {as pointed out by \cite{Ginovyan-Sahakyan-2013}}, the proof of Theorem~2 of \cite{Lieberman-Phillips-2004} 
contains an inaccuracy in the part that the power counting theorem developed by \cite{Fox-Taqqu-1987} was used, see Remark~3 of \cite{Ginovyan-Sahakyan-2013} for more details. 
To the best of the author's knowledge, the proof of (\ref{Error_Lieberman-Phillips}) has not been completed so far in the literature despite {its importance in application to some studies of statistics and econometrics}, 
e.g. see \cite{Lieberman-2005}, \cite{Lieberman-Phillips-2005}, \cite{Lieberman-Rosemarin-Rousseau-2012} 
and their related literature as well as our recent paper~\cite{Takabatake-2022-QWLE}. 
The aim of {the present paper} is 
to prove (\ref{Error_Lieberman-Phillips}) for $2\pi$-periodic functions $g\in{F}_{\alpha}^{(1)}$ and $h\in{F}_{\beta}^{(1)}$ with $\alpha,\beta\in(-\infty,1)$ under (\ref{para_rest1}) and to fill the gap in the literature. 
Other results related to the convergence of $\bar{E}_{n}^{p}$ 
can be found in \cite{Lieberman-Phillips-2004} and {a} comprehensive survey paper~\cite{Ginovyan-Sahakyan-Taqqu-2014}. 

The present paper is organized as follows. 
We introduce some additional notation in Section~\ref{Section_notation}. Our main result is summarized in Section~\ref{Section_main_results}. 
Preliminary lemmas used in the proof of the main result are given in Section~\ref{Section_preliminary_lemmas}. 
The proof of the main result is given in Section~\ref{Section_proof_main_result}.
\section{Notation}\label{Section_notation}
	Define by $D_{n}(x):=\sum_{t=1}^{n}e^{\sqrt{-1}xt}$ for $x\in\mathbb{R}$. 	
	Set $\Theta:=(-\infty,1)$ and define by $\psi_{p}(\theta):=\sum_{j=1}^{p}(\alpha_{j}+\beta_{j})$ and $\overline{\psi}_{p}(\theta):=\sum_{j=1}^{p}(\alpha_{j}+\beta_{j})_{+}$ for $\theta=(\alpha_{1},\cdots,\alpha_{p},\beta_{1},\cdots,\beta_{p})\in\Theta^{2p}$.
	Note that $\psi_{p}(\theta)_{+}\leq\overline{\psi}_{p}(\theta)$ holds. 
	Here we introduce 
	the restricted (parameter) space of $\Theta^{2p}$ defined by 
	\begin{equation*}
		\Theta_{1}^{2p}:=\left\{\theta\in\Theta^{2p}:\overline{\psi}_{p}(\theta)<1\right\}. 
	\end{equation*}
	Obviously $\max_{j=1,\cdots,p}(\alpha_{j}+\beta_{j})_{+}\leq\overline{\psi}_{p}(\theta)<1$ for any $\theta=(\alpha_{1},\cdots,\alpha_{p},\beta_{1},\cdots,\beta_{p})\in\Theta^{2p}_{1}$.

	For non-negative sequences $\{a_{n}\}_{n\in\mathbb{N}}$ and $\{b_{n}\}_{n\in\mathbb{N}}$, we write $a_{n}\lesssim b_{n}$ if there exists a constant $C>0$ such that $a_{n}\leq Cb_{n}$ for sufficiently large $n$. 
	For a set $\Theta_{\ast}$ of $\mathbb{R}^{q}$ and sequences of positive functions $\{a_{n}(\theta)\}_{n\in\mathbb{N}}$ and $\{b_{n}(\theta)\}_{n\in\mathbb{N}}$ on $\Theta_{\ast}$, we write $a_{n}(\theta)\lesssim_{u} b_{n}(\theta)$ (resp.~$a_{n}(\theta)=o_{u}(b_{n}(\theta))$ as $n\to\infty$) {compact} uniformly on $\Theta_{\ast}$ if $\sup_{\theta\in\mathcal{K}}|a_{n}(\theta)/b_{n}(\theta)|\lesssim 1$ (resp.~$\sup_{\theta\in\mathcal{K}}|a_{n}(\theta)/b_{n}(\theta)|=o(1)$ as $n\to\infty$) for any compact subset $\mathcal{K}$ of $\Theta_{\ast}$. 
	Moreover, for a set $A$ of $\mathbb{R}^{r}$ and sequences of functions $\{a_{n}(x,\theta)\}_{n\in\mathbb{N}}$ and $\{b_{n}(x,\theta)\}_{n\in\mathbb{N}}$ on $\mathbb{R}^{r}\times\Theta_{\ast}$ which are always positive on $A\times\Theta_{\ast}$, we write $a_{n}(x,\theta)\lesssim_{u} b_{n}(x,\theta)$ uniformly on $A\times\Theta_{\ast}$ if $\sup_{x\in A}|a_{n}(x,\theta)/b_{n}(x,\theta)|\lesssim_{u} 1$ {compact} uniformly on $\Theta_{\ast}$. 
	For notational simplicity, we omit the word ``{compact} uniformly on $\Theta_{\ast}$'' when $\Theta_{\ast}=\Theta_{1}^{2p}$. 
	
	Finally, we introduce the following function spaces $F$ and $F^{(1)}$:
	\begin{align*}
		&F:=\left\{f\in L^{1}(\Pi\times\Theta):~|x|^{\alpha}|f(x,\alpha)|\lesssim_{u} 1~\mbox{uniformly on $\Pi_{0}\times\Theta$}\right\},\\ 
		&F^{(1)}:=\left\{f\in F\cap C^{1}_{0}(\Pi\times\Theta):|x|^{\alpha+1}\left|\frac{\mathrm{d}f}{\mathrm{d}x}(x,\alpha)\right|\lesssim_{u} 1~\mbox{uniformly on $\Pi_{0}\times\Theta$}\right\},
	\end{align*}
	where $L^{1}(\Pi\times\Theta)$ (resp.~$C^{1}_{0}(\Pi\times\Theta)$) denotes the set of functions $f(x,\alpha)$ on $\Pi\times\Theta$ such that ${x}\mapsto f(x,\alpha)$ is integrable on $\Pi$~(resp.~{continuously} differentiable on $\Pi_{0}$) for each $\alpha\in\Theta$. 
	\section{Main Results}\label{Section_main_results}
	In the present paper, we will prove the slightly generalized version of (\ref{Error_Lieberman-Phillips}) stated in the following theorem. 
	\begin{theorem}\label{Thm_LP04_ext}
		Let $p\in\mathbb{N}$. For all $j=1,\cdots,p$, we assume $g_{j},h_{j}\in{F}^{(1)}$ and $x\mapsto g_{j}(x,\alpha)$ and $x\mapsto h_{j}(x,\beta)$ can be extended to periodic functions on $\mathbb{R}$ with period $2\pi$ for each $\alpha,\beta\in\Theta$. 
		Then
		\begin{equation*}
	   		E_{n}^{p}(\theta):=\left|\frac{1}{n}\mathrm{Tr}\left[\prod_{j=1}^{p}T_{n}(g_{j})T_{n}(h_{j})\right]
	   		-(2\pi)^{2p-1}\int_{\Pi}\prod_{j=1}^pg_{j}(x)h_{j}(x)\,\mathrm{d}x\right|=o_{u}(n^{-1+\overline{\psi}_{p}(\theta)+\epsilon})
		\end{equation*}
		as $n\to\infty$ for any $\epsilon>0$, where we denote by $g_{j}(x)\equiv g_{j}(x,\alpha_{j})$ and $h_{j}(x)\equiv h_{j}(x,\beta_{j})$ for $\alpha_{j},\beta_{j}\in\Theta$ for notational simplicity.
	\end{theorem}
	\section{Preliminary Lemmas}\label{Section_preliminary_lemmas}
	First we summarize useful properties of $D_{n}(x)$ used in the proof of Theorem~\ref{Thm_LP04_ext} without proofs. 
	See 
	\cite{Fox-Taqqu-1987}, \cite{Dahlhaus-1989} and \cite{Dahlhaus-2006} for more details. 
	\begin{lemma}\label{lemma_Dn}
		\begin{enumerate}[$(1)$]
			\item\label{lemma_Dn1} 
			$D_{n}(0)=n$, $D_{n}(-x)=D_{n}(x)^{\ast}$ for any $x\in\Pi$ and $D_{n}(x)$ is $2\pi$-periodic.
			\item\label{lemma_Dn2}
			$\int_{\Pi}D_{n}(x-y)D_{n}(y-z)\,\mathrm{d}y={2\pi}D_{n}(x-z)$ for any $x,z\in\Pi$. 
			\item\label{lemma_Dn3}
			Set $L_{n}(x):=\min(|x|^{-1},n)$. Then we have 
			$|D_{n}(x)|\lesssim L_{n}(x)$ for any $x\in\Pi$ and $L_{n}(x)$ satisfies the following properties:
			\begin{enumerate}[$(a)$]
			 \item $L_{n}(x)\leq\min\{n,2n(1+n|x|)^{-1},n^{\eta}|x|^{\eta-1}\}$ for any $\eta\in[0,1]$ and $x\in\Pi$.
			 \item 
			 $\int_{\Pi}L_{n}(x-y)L_{n}(y)\,\mathrm{d}y\lesssim L_{n}(x)\log{n}$ for any $x\in\Pi$.
			\end{enumerate}
		\end{enumerate}
	\end{lemma}
	By a straightforward calculation, we can also prove the following result, see Appendix~\ref{Appendix_Lemma} for the proof.
	\begin{lemma}\label{lem_err_rep}
		Let $p\in\mathbb{N}$. For all $j=1,\cdots,p$, we assume $g_{j}\in{F}_{\alpha_{j}}$, $h_{j}\in{F}_{\beta_{j}}$ for $\alpha_{j},\beta_{j}\in\Theta$, where the function space ${F}_{\alpha}$ is defined by $(\ref{F-alpha})$, and $g_{j}$ and $h_{j}$ can be extended to periodic functions on $\mathbb{R}$ with period $2\pi$. 
		Then		\begin{align*}
			\mathrm{Tr}\left[\prod_{j=1}^pT_n(g_{j})T_n(h_{j})\right]
			=&\int_{\Pi^{2p}}\prod_{j=1}^{p}g_{j}(y_{2j-1})h_{j}(y_{2j})\cdot\prod_{j=1}^{2p}D_{n}(y_{j}-y_{j-1})\,\mathrm{d}y_{1}\cdots\,\mathrm{d}y_{2p}\\
			=&\int_{\Pi^{2p}}\left(\prod_{j=1}^{p}g_{j}(\overline{x}_{2j-1})h_{j}(\overline{x}_{2j})\right)D_{n}(\overline{x}_{2p}-x_{1})^{\ast}\prod_{j=2}^{2p}D_{n}(x_{j})\,\mathrm{d}x_{1}\cdots\,\mathrm{d}x_{2p},
		\end{align*}
		where we denote by $y_{0}:=y_{2p}$ and by $\overline{x}_{j}:=\sum_{k=1}^{j}x_{k}$ for $j=1,\cdots,2p$ for notational simplicity. 
	\end{lemma}
	\section{Proof of Theorem~\ref{Thm_LP04_ext}}\label{Section_proof_main_result}
	\subsection{Outline of Proof of Theorem~\ref{Thm_LP04_ext}}
	First we introduce some notation used in the proof of Theorem~\ref{Thm_LP04_ext}. 
	Fix $p\in\mathbb{N}$ and write $E_{n}(\theta):=E_{n}^{p}(\theta)$ for notational simplicity.  
	Set
	\begin{equation*}
		I(\theta):=(2\pi)^{2p-1}\int_{\Pi}\prod_{j=1}^pg_{j}(x)h_{j}(x)\,\mathrm{d}x,\ \ \theta\in\Theta_{1}^{2p}.
	\end{equation*}
	Fix $c>1$ 
	and define by
	\begin{align*}
		&W_{j}:=
			\left\{\mathbf{x}=(x_{1},\cdots,x_{2p})\in\mathbb{R}^{2p}:|\overline{x}_{j}|\leq c|x_{j+1}|\right\},\ \ j=1,\cdots,2p-1,\\
		&W_{2p}:=
			\left\{\mathbf{x}=(x_{1},\cdots,x_{2p})\in\mathbb{R}^{2p}:|\overline{x}_{2p}|\leq c|\overline{x}_{2p}-x_{1}|\right\}.
	\end{align*}
	Set $W:=\bigcup_{j=1}^{2p}W_{j}$, $k_{j}^{(0)}(x):=g_{j}(x)h_{j}(x)$ and $k_{j}^{(1)}(x):=g_{j}(x)(h_{j}(x+x_{2j})-h_{j}(x))$ for $j=1,\cdots,p$, and $I_{n}^{\pi}(\theta):=I_{n,1}^{\pi}(\theta)+I_{n,2}^{\pi}(\theta)$ is defined by
	\begin{align*}
		&I_{n,1}^{\pi}(\theta):=
		\int_{\Pi^{2p}\cap{W}}\prod_{j=1}^{p}k_{j}^{(\pi_{j})}(\overline{x}_{2j-1})\cdot D_{n}(\overline{x}_{2p}-x_{1})^{\ast}\prod_{j=2}^{2p}D_{n}(x_{j})\,\mathrm{d}x_{1}\cdots\,\mathrm{d}x_{2p},\\
		&I_{n,2}^{\pi}(\theta):=\int_{\Pi^{2p}\cap{W}^{c}}\prod_{j=1}^{p}k_{j}^{(\pi_{j})}(\overline{x}_{2j-1})\cdot D_{n}(\overline{x}_{2p}-x_{1})^{\ast}\prod_{j=2}^{2p}D_{n}(x_{j})\,\mathrm{d}x_{1}\cdots\,\mathrm{d}x_{2p}
	\end{align*}
	for $\pi=(\pi_{1},\cdots,\pi_{p})\in\{0,1\}^{p}$. 
	Write $I_{n}(\theta):=I_{n}^{(0,\cdots,0)}(\theta)$ for notational simplicity. 
	
	Since $g_{j}(\overline{x}_{2j-1})h_{j}(\overline{x}_{2j})=k_{j}^{(0)}(\overline{x}_{2j-1})+k_{j}^{(1)}(\overline{x}_{2j-1})$, we obtain the following upper bound for $E_{n}(\theta)$ using Lemma~\ref{lem_err_rep}:
	\begin{equation}\label{En_upper_bound}
		E_{n}(\theta)\lesssim_{u}n^{-1}|I_{n}(\theta)-nI(\theta)| +n^{-1}\sum_{\pi=(\pi_{1},\cdots,\pi_{p})\in\{0,1\}^{p}\setminus\{(0,\cdots,0)\}}|I_{n}^{\pi}(\theta)|.
	\end{equation}
	Thanks to (\ref{En_upper_bound}), Theorem~\ref{Thm_LP04_ext} follows once we have proven the following results.
	\begin{proposition}\label{key_proposition}
		Suppose the assumptions given in Theorem~$\ref{Thm_LP04_ext}$ {to hold}.
		\begin{enumerate}[$(1)$]
			\item\label{key_proposition_W}
			For any $\epsilon>0$ and $\pi\in\{0,1\}^{p}$, we have $I_{n,1}^{\pi}(\theta)=o_{u}(n^{\overline{\psi}_{p}(\theta)+\epsilon})$ as $n\to\infty$.
			\item\label{key_proposition_Wc}
			For any $\pi\in\{0,1\}^{p}\setminus\{(0,\cdots,0)\}$, we have $I_{n,2}^{\pi}(\theta)=o_{u}(n^{\psi_{p}(\theta)_{+}}(\log{n})^{2p})$ as $n\to\infty$.
			\item\label{key_proposition_main_term}
			For any $\epsilon>0$, we have $|I_{n}(\theta)-nI(\theta)|=o_{u}(n^{\overline{\psi}_{p}(\theta)+\epsilon})$ as $n\to\infty$.
		\end{enumerate}
	\end{proposition}
	\begin{remark}\rm 
		Our decomposition of the integration domain $\Pi^{2p}$ into the disjoint sets $\Pi^{2p}\cap{W}$ and $\Pi^{2p}\cap{W}^{c}$ is different from what the previous studies~\cite{Fox-Taqqu-1987} and \cite{Lieberman-Phillips-2004} used. 
		We firstly recall their decomposition of the integration domain and explain several connections between these decompositions. 
		In the proofs of Theorem~1 of \cite{Fox-Taqqu-1987} and Theorem~2 of \cite{Lieberman-Phillips-2004}, they decompose $\Pi^{2p}$ into the following three disjoint sets:
		\begin{equation*}
			E_{t}:=U_{\pi}\setminus\{\widetilde{W}^{\prime}\cup{U}_{t}\},\ \ F_{t}:=U_{t}\setminus\widetilde{W}^{\prime},\ \ G:=U_{\pi}\cap\widetilde{W}^{\prime}
		\end{equation*}
		for $t\in(0,\pi]$, where $U_{t}:=[-t,t]^{2p}$ and $\widetilde{W}^{\prime}:=\bigcup_{j=1}^{2p}\widetilde{W}_{j}^{\prime}$ is defined by
		\begin{equation*}
			\widetilde{W}_{j}^{\prime}:=\left\{\mathbf{y}=(y_{1},\cdots,y_{2p})\in\mathbb{R}^{2p}:|y_{j}|\leq\frac{|y_{j+1}|}{2}\right\},\ \ j=1,\cdots,2p,
		\end{equation*}
		where we write $y_{2p+1}:=y_{1}$ for notational simplicity. 
		By the change of variables $x_{1}=y_1$ and $x_{j}=y_{j}-y_{j-1}$ for $j=2,\cdots,2p$, the domain $W$ is transformed to $\widetilde{W}$ defined by 
		\begin{equation*}
			\widetilde{W}:=\bigcup_{j=1}^{2p}\widetilde{W}_{j},\ \ \widetilde{W}_{j}:=\left\{(y_1,\cdots,y_{2p})\in\mathbb{R}^{2p}:|y_{j}|\leq c|y_{j+1}-y_{j}|\right\},\ \ j=1,\cdots,2p.
		\end{equation*}
		Note that the Jacobian determinant of the above trasformation on $\mathbb{R}^{2p}$ to itself is equal to one. 
		If $c\geq 1$, we have $\widetilde{W}_{j}^{\prime}\subsetneq\widetilde{W}_{j}$ for each $j=1,\cdots,2p$ so that we obtain $\Pi^{2p}\cap\widetilde{W}^{c}\subsetneq E_{t}\cup{F}_{t}=\Pi^{2p}\setminus\widetilde{W}^{\prime}$ and $G\subsetneq\Pi^{2p}\cap\widetilde{W}$. 
		
		Let $\alpha_{1}=\alpha_{2}=\cdots=\alpha_{p}$ and $\beta_{1}=\beta_{2}=\cdots=\beta_{p}$. 
		Taking Remark~3 of \cite{Ginovyan-Sahakyan-2013} into account, it may be impossible to prove
		\begin{align*}
			&\int_{F_{t}}\left|\prod_{j=1}^{p}g_{j}(y_{2j-1})h_{j}(y_{2j})\prod_{j=1}^{2p}D_{n}(y_{j}-y_{j+1})\right|\,\mathrm{d}y_{1}\cdots\,\mathrm{d}y_{2p}=o(n^{\overline{\psi}_{p}(\theta)+\epsilon})\ \ \mbox{as $n\to\infty$},\\
			&\int_{F_{t}}\left|\prod_{j=1}^{p}g_{j}(y_{1})h_{j}(y_{1})\prod_{j=1}^{2p}D_{n}(y_{j}-y_{j+1})\right|\,\mathrm{d}y_{1}\cdots\,\mathrm{d}y_{2p}=o(n^{\overline{\psi}_{p}(\theta)+\epsilon})\ \ \mbox{as $n\to\infty$}
		\end{align*}
		for any $\epsilon>0$ as stated in the proof of Theorem~2 of \cite{Lieberman-Phillips-2004} 
		because we can not obtain ``good'' dominating functions of the integrands on $F_{t}$ for $t\in(0,\pi]$. 
		On the other hand, 
		we can prove Theorem~\ref{Thm_LP04_ext} using our decomposition of the integration domain $\Pi^{2p}$ because we can prove Proposition~\ref{key_proposition}~(\ref{key_proposition_W}) using Lemma~\ref{lemma_Dn} and the power counting theorem by \cite{Fox-Taqqu-1987}, and obtain a dominating function of the integrand in $I_{n,2}^{\pi}(\theta)$ which is good enough to prove Proposition~\ref{key_proposition}~(\ref{key_proposition_Wc}), see Section~\ref{Section_proof_key_prop1} and Section~\ref{Section_proof_key_prop2} respectively for more details. 
	\end{remark}
	\subsection{Proof of Proposition~\ref{key_proposition}~(\ref{key_proposition_W})}\label{Section_proof_key_prop1}
	For any $\epsilon>0$ and $\theta_{j}=(\alpha_{j},\beta_{j})\in\Theta^{2}$, $j=1,\cdots,p$, let us take $\eta_{2j-1}\equiv\eta_{2j-1}(\theta_{j},\epsilon)$ and $\eta_{2j}\equiv\eta_{2j}(\theta_{j},\epsilon)$ satisfying
	\begin{equation*}
		(\alpha_{j}+\beta_{j})_{+}<\eta_{2j-1}+\eta_{2j}<(\alpha_{j}+\beta_{j})_{+}+\frac{\epsilon}{p},\ \ \eta_{2j-1},\eta_{2j}\in(0,1),
	\end{equation*}
	for all $j=1,\cdots,p$. 
	Note that we have $\overline{\psi}_{p}(\theta)<\sum_{j=1}^{2p}\eta_{j}<\overline{\psi}_{p}(\theta)+\epsilon$ for $\theta=(\alpha_{1},\cdots,\alpha_{p},\beta_{1},\cdots,\beta_{p})$. 
	First we can prove
	\begin{equation*}
		|I_{n,1}^{\pi}(\theta)|\lesssim_{u} n^{\sum_{j=1}^{2p}\eta_{j}}\sum_{i=1}^{2p}\int_{\Pi^{2p}\cap{W}_{i}}f_{n}(\mathbf{x})\,\mathrm{d}\mathbf{x}
	\end{equation*}
	using Lemma~\ref{lemma_Dn}, where $\mathbf{x}=(x_{1},\cdots,x_{2p})$ and 
	\begin{equation*}
		f_{n}(\mathbf{x}):=\prod_{j=1}^p|\overline{x}_{2j-1}|^{-\alpha_{j}}
		{\left(|\overline{x}_{2j}|^{-\beta_{j}}+|\overline{x}_{2j-1}|^{-\beta_{j}}\right)}\cdot\prod_{j=1}^{2p}|\overline{x}_{j+1}-\overline{x}_{j}|^{\eta_{j}-1}.
	\end{equation*}
	Here we denote by $\overline{x}_{2p+1}:=x_{1}$ for notational simplicity. 
	In the similar way to the proof of Proposition~6.1 of \cite{Fox-Taqqu-1987}, we can also prove
	\begin{equation*}
		\sum_{i=1}^{2p}\int_{\Pi^{2p}\cap{W}_{i}}{f}_{n}(\mathbf{x})\,\mathrm{d}\mathbf{x}\lesssim_{u}1.
	\end{equation*}
	This completes the proof.
	\subsection{Proof of Proposition~\ref{key_proposition}~(\ref{key_proposition_Wc})}\label{Section_proof_key_prop2}
	Before proving Proposition~\ref{key_proposition}~(\ref{key_proposition_Wc}), we prepare the following lemma. 
	Denote by the hyperplane $H_{1}:=\{(x_{1},\cdots,x_{2p})\in\mathbb{R}^{2p}:x_{1}=0\}$ and set $A:=\Pi^{2p}\cap{W}^{c}\cap{H}_{1}^{c}$.
	\begin{lemma}\label{key_lemma_W-tilda-c}
		Let $p\in\mathbb{N}$. Assume $g_{j}\in{F}$ and $h_{j}\in{F}^{(1)}$ for all $j=1,\cdots,p$. 
		Then for any 
		$\pi=(\pi_1,\cdots,\pi_p)\in\{0,1\}^p$,
		\begin{equation*}
			\left|\prod_{j=1}^pk_j^{(\pi_j)}(\bar{x}_{2j-1})\right|\lesssim_{u}|x_{1}|^{-\psi_{p}(\theta)-|\pi|}\prod_{j=1}^p|x_{2j}|^{\pi_{j}}
		\end{equation*}
		uniformly on $A\times\Theta_{1}^{2p}$, where we write $|\pi|:=\sum_{j=1}^{p}\pi_{j}$.
	\end{lemma}
	\begin{proof}
		Fix $(\pi_{1},\cdots,\pi_{p})\in\{0,1\}^{p}$. We can assume $\overline{x}_{j}\neq 0$ for all $j=1,\cdots,2p$ without loss of generality. Indeed, $(x_{1},\cdots,x_{2p})\in{W}^{c}$ implies
		\begin{equation}\label{key_inequality_21}
			c_{-}|\overline{x}_{j-1}|\leq|\overline{x}_{j}|\leq c_{+}|\overline{x}_{j-1}|,\ \ j=2,3,\cdots,2p,
		\end{equation} 
		where $c_{\pm}:=1\pm c^{-1}$.  
		Note that $c_\pm>0$ thanks to $c>1$. Therefore $x_{1}\neq 0$ implies $\overline{x}_{j}\neq 0$ for all $j=1,\cdots,2p$. 
		Moreover, $(x_{1},\cdots,x_{2p})\in{W}^{c}\cap{H}_{1}^{c}$ also implies $|\overline{x}_{j}+u(\overline{x}_{j+1}-\overline{x}_{j})|\geq c_{-}|\overline{x}_{j}|>0$ for any $u\in[0,1]$ and $j=1,\cdots,2p-1$ so that the mean-value theorem gives
		\begin{equation}\label{key_inequality_22}
			\left|h_{j}(\overline{x}_{2j})-h_{j}(\overline{x}_{2j-1})\right|
			\leq|x_{2j}|\sup_{c_{-}|\overline{x}_{2j-1}|\leq x\leq c_{+}|\overline{x}_{2j-1}|}\left|\frac{\mathrm{d}h_{j}}{\mathrm{d}x}(x)\right|
			\lesssim_{u}|x_{2j}||\overline{x}_{2j-1}|^{-\beta_{j}-1}
		\end{equation}
		uniformly on $A\times\Theta_{1}^{2p}$ thanks to (\ref{key_inequality_21}). Therefore we obtain 
		\begin{equation*}
			\left|\prod_{j=1}^{p}k_{j}^{(\pi_{j})}(\overline{x}_{2j-1})\right|
			\lesssim_{u}\prod_{j=1}^p|x_{2j}|^{\pi_{j}}|\overline{x}_{2j-1}|^{-(\alpha_{j}+\beta_{j})-\pi_{j}}
			\lesssim_{u}|x_{1}|^{-\psi_{p}(\theta)-|\pi|}\prod_{j=1}^p|x_{2j}|^{\pi_{j}}
		\end{equation*}
		uniformly on $A\times\Theta_{1}^{2p}$ thanks to (\ref{key_inequality_21}) and (\ref{key_inequality_22}). This completes the proof.
	\end{proof}
	\begin{proof}[Proof of Proposition~$\ref{key_proposition}~(\ref{key_proposition_Wc})$]
		From the assumption, there exists $J\in\{1,\cdots,p\}$ such that $\pi_{J}=1$. 
		Note that Lemma~\ref{key_lemma_W-tilda-c}, the definition of $W$ and (\ref{key_inequality_21}) imply 
		\begin{equation}\label{key_inequality_Wc}
			\left|\prod_{j=1}^pk_j^{(\pi_j)}(\bar{x}_{2j-1})\right|
			\lesssim_{u}|x_{1}|^{-\psi_{p}(\theta)_{+}-1}|x_{2J}|
		\end{equation}
		uniformly on $A\times\Theta_{1}^{2p}$. 
		Since the Lebesgue measure of $H_{1}$ is zero and Lemma~\ref{lemma_Dn} gives
		\begin{equation*}
			\int_{\Pi^{2p-2}}
			\left|D_{n}(\overline{x}_{2p}-x_{1})^{\ast}\prod_{j=2}^{2p}D_{n}(x_{j})\right|\,\mathrm{d}x_{2}\cdots\,\mathrm{d}x_{2J-1}\,\mathrm{d}x_{2J+1}\cdots\,\mathrm{d}x_{2p}\leq(\log{n})^{2p-2}L_{n}(x_{2J})^{2},
		\end{equation*}
		we can show
		\begin{align*}
			\left|I_{n,2}^{\pi}(\theta)\right|
			&\lesssim_{u}\int_{A}|x_{1}|^{-\psi_{p}(\theta)_{+}-1}|x_{2J}|\left|D_{n}(\overline{x}_{2p}-x_{1})^{\ast}\prod_{j=2}^{2p}D_{n}(x_{j})\right|\,\mathrm{d}x_{1}\cdots\,\mathrm{d}x_{2p}\\
			&\lesssim_{u}(\log{n})^{2p-2}\int_{A_{J}}|x_{1}|^{-\psi_{p}(\theta)_{+}-1}|x_{2J}|L_{n}(x_{2J})^{2}\,\mathrm{d}x_{1}\,\mathrm{d}x_{2J}
		\end{align*}
		thanks to (\ref{key_inequality_Wc}), where
		\begin{equation*}
			A_{J}:=\{(x_{1},x_{2J})\in\Pi^{2}:|x_{2J}|<c_{J}|x_{1}|,x_{1}\neq 0\},\ \ c_{J}:=c^{-1}(c_{+})^{2J-2}.
		\end{equation*}		
		Note that we can write
		\begin{equation*}
			\int_{A_{J}}|x_{1}|^{-\psi_{p}(\theta)_{+}-1}|x_{2J}|L_{n}(x_{2J})^{2}\,\mathrm{d}x_{1}\,\mathrm{d}x_{2J}
			=4(I_{1}^{n}+I_{2}^{n}),
		\end{equation*}
		where
		\begin{align*}
			&J_{1}^{n}(\theta):=\int_{0}^{(nc_{J})^{-1}}\,\mathrm{d}x_{1}\,x_{1}^{-\psi_{p}(\theta)_{+}-1}\int_{0}^{c_{J}x_{1}}\,\mathrm{d}x_{2J}\,x_{2J}L_{n}(x_{2J})^{2},\\
			&J_{2}^{n}(\theta):=\int_{(nc_{J})^{-1}}^{\pi}\,\mathrm{d}x_{1}\,x_{1}^{-\psi_{p}(\theta)_{+}-1}\int_{0}^{c_{J}x_{1}}\,\mathrm{d}x_{2J}\,x_{2J}L_{n}(x_{2J})^{2}.
		\end{align*}
		In the rest of the proof, we evaluate $J_{1}^{n}(\theta)$ and $J_{2}^{n}(\theta)$ respectively. First, by a straightforward calculation, we obtain
		\begin{equation*}
			J_{1}^{n}(\theta)
			\lesssim_{u} n^{2}\int_{0}^{(nc_{J})^{-1}}x_{1}^{1-\psi_{p}(\theta)_{+}}\,\mathrm{d}x_{1}
			\lesssim_{u} n^{\psi_{p}(\theta)_{+}}.
		\end{equation*}
		Moreover, if $x_{1}>(nc_{J})^{-1}$, then
		\begin{equation*}
			\int_{0}^{c_{J}x_{1}}x_{2J}L_{n}(x_{2J})^{2}\,\mathrm{d}x_{2J}
			\lesssim 1+\int_{{\frac{1}{n}}}^{c_{J}x_{1}}x_{2J}^{-1}\,\mathrm{d}x_{2J}
			\lesssim |\log{x_{1}}|+\log{n}.
		\end{equation*}
		Therefore we also obtain
		\begin{align*}
			J_{2}^{n}(\theta)
			&\lesssim_{u} \int_{(nc_{J})^{-1}}^{\pi}x_{1}^{-\psi_{p}(\theta)_{+}-1}(|\log{x_{1}}|+\log{n})\,\mathrm{d}x_{1}\\
			&\lesssim_{u} n^{\psi_{p}(\theta)_{+}}\int_{1}^{nc_{J}\pi}u^{-\psi_{p}(\theta)_{+}-1}(\log{u}+\log{n})\,\mathrm{d}u\\
			&\lesssim_{u} n^{\psi_{p}(\theta)_{+}}\int_{1}^{nc_{J}\pi}u^{-1}(\log{u}+\log{n})\,\mathrm{d}u\lesssim_{u}n^{\psi_{p}(\theta)_{+}}(\log{n})^{2}.
		\end{align*}
		This completes the proof. 
	\end{proof}
	\subsection{Proof of Proposition~\ref{key_proposition}~(\ref{key_proposition_main_term})}
	Fix $\epsilon>0$. Denote by $k_{j}:=k_{j}^{(0)}$ for notational simplicity. First note that we can write
	\begin{equation*}
		I_{n}(\theta)
		=(2\pi)^{p}\int_{\Pi^{p}}\prod_{j=1}^{p}k_{j}(y_{j})D_{n}(y_{j}-y_{j+1})\,\mathrm{d}y_{1}\cdots\,\mathrm{d}y_{p}
	\end{equation*}
	using Lemmas~\ref{lemma_Dn} and \ref{lem_err_rep}. Set
	\begin{equation*}
		s_{J}(y):=\prod_{j=p-J+1}^{p}k_{j}(y),\ \ \overline{s}_{J}(y):=s_{J}(y_{p-J+1})-s_{J}(y)
	\end{equation*}
	for $J=1,\cdots,p$. Then we can show
	\begin{align*}
		&I_{n}(\theta)\\
		&=(2\pi)^{p}\int_{\Pi^{p}}\left[s_{1}(y_{p-1})+\overline{s}_{1}(y_{p-1})\right]\prod_{j=1}^{p-1}k_{j}(y_{j})\prod_{j=1}^{p}D_{n}(y_{j}-y_{j+1})\,\mathrm{d}y_1\cdots\,\mathrm{d}y_{p}\\
		&=(2\pi)^{p+1}\int_{\Pi^{p-1}}s_{2}(y_{p-1})\prod_{j=1}^{p-2}k_{j}(y_{j})\cdot D_{n}(y_{p-1}-y_{1})\prod_{j=1}^{p-2}D_{n}(y_{j}-y_{j+1})\,\mathrm{d}y_{1}\cdots\,\mathrm{d}y_{p-1}+o_{u}(n^{\overline{\psi}_{p}(\theta)+\epsilon})
	\end{align*}
	as $n\to\infty$ using Lemma~\ref{lemma_Dn} and Proposition~\ref{key_proposition}~(\ref{key_proposition_W})-(\ref{key_proposition_Wc}). 
	Therefore we can easily prove
	\begin{equation*}
		I_{n}(\theta)
		=(2\pi)^{2p-1}n\int_{\Pi}s_{p}(y_{1})\,\mathrm{d}y_{1}+o_{u}(n^{\overline{\psi}_{p}(\theta)+\epsilon})\ \ \mbox{as $n\to\infty$}
	\end{equation*}
	using Lemma~\ref{lemma_Dn} and Proposition~\ref{key_proposition}~(\ref{key_proposition_W})-(\ref{key_proposition_Wc}) repeatedly.  
	This completes the proof.
\bibliographystyle{acmtrans-ims}
\bibliography{myref_tkbtk}

\appendix
\section{Proof of Lemma~\ref{lem_err_rep}}\label{Appendix_Lemma}
By a straightforward calculation, we can write
	\begin{align*}
		T_{n}(\theta)&:=\mathrm{Tr}\left[\prod_{j=1}^pT_n(g_{j})T_n(h_{j})\right]\\
		&=\sum_{t_{1},\cdots,t_{2p}=1}^n\prod_{j=1}^p\widehat{g_{j}}(t_{2j-1}-t_{2j})\widehat{h_{j}}(t_{2j}-t_{2j+1})\\
		&=\int_{\Pi^{2p}}\prod_{j=1}^{p}g_{j}(y_{2j-1})h_{j}(y_{2j})\cdot\prod_{j=1}^{2p}D_{n}(y_{j}-y_{j-1})\,\mathrm{d}y_{1}\cdots\,\mathrm{d}y_{2p},
	\end{align*}
	where we denote by $t_{2p+1}:=t_{1}$ for notational simplicity, because we can show
	\begin{equation*}
		\sum_{t_{1},\cdots,t_{2p}=1}^{n}\prod_{j=1}^{2p}e^{\sqrt{-1}(t_{j}-t_{j+1})y_{j}}
		=\sum_{t_{1},\cdots,t_{2p}=1}^{n}\prod_{j=1}^{2p}e^{\sqrt{-1}(y_{j}-y_{j-1})t_{j}}
		=\prod_{j=1}^{2p}D_{n}(y_{j}-y_{j-1}).
	\end{equation*}
	Moreover, by the change of variables $x_{1}=y_1$ and $x_{j}=y_j-y_{j-1}$ for $j=2,\cdots,2p$, we can rewrite
	\begin{align*}
		T_{n}(\theta)&=\int_{[-\pi,\pi]}\,\mathrm{d}x_{1}\int_{[-\pi-x_{1},\pi-x_{1}]}\,\mathrm{d}x_2\cdots\int_{[-\pi-\overline{x}_{2p-1},\pi-\overline{x}_{2p-1}]}\,\mathrm{d}x_{2p}\\
		&\hspace{2cm}\prod_{j=1}^pg_{j}(\overline{x}_{2j-1})h_{j}(\overline{x}_{2j})D_{n}(\overline{x}_{2p}-x_{1})^\ast\prod_{j=2}^{2p}D_{n}(x_{j})\\
		&=\int_{\Pi^{2p}}\prod_{j=1}^pg_{j}(\overline{x}_{2j-1})h_{j}(\overline{x}_{2j})D_{n}(\overline{x}_{2p}-x_{1})^\ast\prod_{j=2}^{2p}D_{n}(x_{j})\,\mathrm{d}x_{1}\cdots\,\mathrm{d}x_{2p}
	\end{align*}
	since $y_j=y_1+\sum_{r=2}^{j}(y_{r}-y_{r-1})=\sum_{r=1}^{j}x_r=\bar{x}_{j}$, where we used the $2\pi$-periodicity of $D_{n}$, $g_{j}$ and $h_{j}$ for $j=1,\cdots,p$ in the last equality. 
	This completes the proof.
	
\end{document}